\documentclass[a4paper, 12pt]{amsart}
\usepackage{amsmath, amssymb, amscd, enumerate}

\newtheorem{theorem}{Theorem}[section]
\newtheorem{proposition}[theorem]{Proposition}
\newtheorem{lemma}[theorem]{Lemma}
\newtheorem{corollary}[theorem]{Corollary}

\newcommand{\sm}{\left(\begin{smallmatrix}}
\newcommand{\esm}{\end{smallmatrix}\right)}
\theoremstyle{remark}
\newtheorem{remark}{Remark}

\newcommand{\SL}{\mathrm{SL}}

\newcommand{\Mp}{\mathrm{Mp}}

\begin{document}

\title[Zagier duality for harmonic weak Maass forms of integral weight]
{Zagier duality for harmonic weak Maass forms of integral weight}

\author{Bumkyu Cho}
\address{Department of Mathematics, Pohang University of Science and
Technology, San 31, Hyoja-dong, Nam-gu, Pohang-si, Gyeongsangbuk-do
790-784, Republic of Korea}

\email{bam@math.kaist.ac.kr}

\author{YoungJu Choie}

\address{Department of Mathematics, Pohang Mathematics Institute,
POSTECH, Pohang, Korea}

\email{yjc@postech.ac.kr}

\thanks{The first author is partially supported by BK21 at POSTECH
and Tae-Joon Park POSTECH Postdoctoral Fellowship and NRF
2010-0008426, and the second author partially supported by NRF20090083909 and NRF2009-0094069}


\dedicatory{}

\begin{abstract}
We show the existence of ``Zagier duality'' between vector valued
harmonic weak Maass forms and vector valued weakly holomorphic
modular forms of integral weight. This duality phenomenon arises
naturally in the context of harmonic weak Maass forms as developed
in recent works by Bruinier, Funke, Ono, and Rhoades \cite{BF,BOR}.
Concerning the isomorphism between the spaces of scalar and vector
valued harmonic weak Maass forms of integral weight, Zagier duality
between scalar valued ones is derived.
\end{abstract}

\maketitle

\section{\bf{Introduction and statement of results}}

For an integer $k$, let  $M_{k +
\frac{1}{2}}^{!+}$ be the space of weakly holomorphic
modular forms $f(\tau)$ of weight $k + \frac{1}{2}$ for
$\Gamma_0(4)$ satisfying the Kohnen's plus space condition, that
is, $f(\tau)$ has a Fourier expansion of the form
\[ f(\tau) = \sum_{(-1)^k n \equiv 0, 1 \bmod 4} c_f(n) q^n. \]

For $d \geq 0$ with $d \equiv 0, 3 \bmod 4$ there is a unique
modular form $f_d \in M_{\frac{1}{2}}^{!+}$ having a
Fourier expansion of the form
\[ f_d(\tau) = q^{-d} + \sum_{D > 0} c_{f_d}(D) q^D. \]

On the other hand, for $D > 0$ with $D \equiv 0, 1 \bmod 4$ there
is also a unique modular form $g_D \in
M_{\frac{3}{2}}^{!+}$ having a Fourier expansion of
the form
\[ g_D(\tau) = q^{-D} + \sum_{d \geq 0} c_{g_D}(d) q^d. \]

As proven by Zagier \cite[Theorem 1]{Za} the $g_1(\tau)$ is
essentially the generating function for the traces of singular
moduli. He also proved the so-called ``Zagier duality''
\cite[Theorem 4]{Za} relating the Fourier coefficients of
$f_d(\tau)$ and $g_D(\tau)$:

\begin{theorem}(Zagier)\label{Za}
\[c_{f_d}(D) = -c_{g_D}(d) \quad \mbox{for all $D$ and $d$}. \]
\end{theorem}

\medskip

Zagier's results have inspired vast research subjects (for instance
see \cite{BO, DJ1, DJ2,  FO, Gu, Gu2, Kim, RO}) and have been extended to study duality properties  on the space of harmonic weak Maass forms
$H_k(\Gamma_0(N)).$  For instance, in terms of the weight (higher
weight) and space (weak Maass forms) aspects, K. Bringmann and K.
Ono \cite[Theorem 1.1]{BO} showed Zagier duality between certain
Maass-Poincar\'e series of weight $-k + 3/2$ and Poincar\'e series
of weight $k + 1/2$ for $\Gamma_0(4)$ with $k \geq 1$. In fact,
their Fourier coefficients are traces of singular moduli of certain
weight $0$ Maass forms. A. Folsom and K. Ono \cite{FO} found Zagier
duality between certain weight $1/2$ harmonic weak Maass forms and
weight $3/2$ weakly holomorphic modular forms on $\Gamma_0(144)$
with Nebentypus $(\frac{12}{\cdot})$. The holomorphic part
of their initial Maass form is essentially Ramanujan's mock theta
function $f(q)$. Also, C. H. Kim \cite{Kim} proved Zagier duality between
 certain weakly holomorphic modular forms of weight $1/2$ and $3/2$ on
$\Gamma_0(4p)$. Here $p$ is a prime such that the genus of the Fricke group
$\Gamma_0(p)^\ast$ equals $0$. In his result, those Fourier coefficients are essentially traces of singular moduli of the Hauptmodul $j_p^\ast(\tau)$ for
$\Gamma_0(p)^\ast$.

Now we introduce the results of integral weight cases. J. Rouse \cite[Theorem 1]{RO}
proved Zagier duality between certain weakly holomorphic modular forms
of weight $0$ and $2$ on $\Gamma_0(p)$ with Nebentypus $(\frac{\cdot}{p})$
where $p = 5, 13, 17$. D. Choi \cite[Theorem 1.2]{Choi}
gave a simple proof of Rouse's result by using the residue theorem, and extended
to any odd prime level $p$. Note that their results concern weakly holomorphic
modular forms. Recently P. Guerzhoy \cite[Theorem 1]{Gu} showed that there
is Zagier duality between certain harmonic weak Maass forms of weight $k$
and weakly holomorphic modular forms of weight $2 - k$ for $\SL_2(\mathbb Z)$ where $k \leq 0$ is even.

\medskip

The purpose of this paper is to derive Zagier duality for harmonic
weak Maass forms and, as a result, this   extends the previously
known Zagier duality between integral weight forms. To this end,
we first show that duality holds between the space $H_{k, \rho_L}$
of vector valued harmonic weak Maass forms and the space $M_{2-k,
\bar{\rho}_L}^!$ of vector valued weakly holomorphic modular forms
with $k \leq 0$ an integer (Theorem \ref{Theorem - Zagier duality
for vector valued modular forms}). This duality phenomenon arises
naturally in the context of harmonic weak Maass forms as developed
in recent works by Bruinier, Funke, Ono, and Rhoades \cite{BF,
BOR}. By taking $L$ in our result as an even unimodular lattice,
we immediately recover the recent result by Guerzhoy \cite{Gu}
(Corollary \ref{Corollary - Guerzhoy}). Now we take $L$ so that
 a certain space $H_{k}^{\epsilon}(\Gamma_0(p), (\frac{\cdot}{p}))$
of scalar valued harmonic weak Maass forms is isomorphic to $H_{k,
\rho_{L}}$ (Proposition \ref{Proposition - isomorphism for
integral weight case}). Then, as another corollary, for any odd
prime $p$ we have Zagier duality between $H_k^{\epsilon}
(\Gamma_0(p), (\frac{\cdot}{p}))$ and a certain  space $M_{2 -
k}^{!\delta} (\Gamma_0(p), (\frac{\cdot}{p}))$ of scalar valued
weakly holomorphic modular forms (Corollary \ref{Corollary -
Zagier duality for scalar valued modular forms}).

To state our main theorem let $L$ be a non-degenerate even lattice
of signature $(b^+, b^-)$, and $L'$ its dual lattice. We denote the
standard basis elements of the group algebra $\mathbb C[L' / L]$ by
$\mathfrak e_\gamma$ for $\gamma \in L'/L$. Let $\rho_L$ be the Weil
representation associated to the discriminant form $(L'/L, Q)$, and
$\bar{\rho}_L$ its dual representation. We write $M_{k, \rho_L}^!$
for the space of $\mathbb C[L'/L]$-valued weakly holomorphic modular
forms of weight $k$ and type $\rho_L$, and $H_{k, \rho_L}$ for the
space of $\mathbb C[L'/L]$-valued harmonic weak Maass forms of
weight $k$ and type $\rho_L$ (see Section 2). In the following
theorem we obtain Zagier duality between the space $H_{k, \rho_L}$
of vector valued harmonic weak Maass forms and the space $M_{2 - k,
\bar{\rho}_L}^!$ of vector valued weakly holomorphic modular forms
with $k \leq 0$ an integer.

\begin{theorem}\label{Theorem - Zagier duality for vector valued modular forms}
Let $L$ be a non-degenerate even lattice of signature $(b^+, b^-)$,
and $k \leq 0$ an integer such that $2k - b^+ + b^- \equiv 0 \bmod
2$. Let $\alpha, \beta \in L'/L$ and $m \in \mathbb Z - Q(\alpha)$,
$n \in \mathbb Z + Q(\beta)$ with $m, n > 0$. Then there exist
$f_{\alpha, m} \in H_{k, \rho_L}$ and $g_{\beta, n} \in M_{2 - k,
\bar{\rho}_L}^!$ with Fourier expansions of the form
\begin{eqnarray*}
f_{\alpha, m}(\tau) & = & f_{\alpha, m}^-(\tau) + q^{-m} \mathfrak
e_\alpha + (-1)^{k + \frac{b^- - b^+}{2}} q^{-m} \mathfrak
e_{-\alpha} + \sum_{\gamma \in L'/L} \sum_{l \in \mathbb Z +
Q(\gamma) \atop l
\geq 0} c_{f_{\alpha, m}}^+(\gamma, l) q^l \mathfrak e_\gamma \\
g_{\beta, n}(\tau) & = & q^{-n} \mathfrak e_\beta + (-1)^{k +
\frac{b^+ - b^-}{2}} q^{-n} \mathfrak e_{-\beta} + \sum_{\gamma \in
L'/L} \sum_{l \in \mathbb Z - Q(\gamma) \atop l > 0} c_{g_{\beta,
n}}(\gamma, l) q^l \mathfrak e_\gamma
\end{eqnarray*}
such that
\[ c_{f_{\alpha, m}}^+(\beta, n) = -c_{g_{\beta, n}}(\alpha, m). \]
If $k < 0$, then $f_{\alpha, m}$, $g_{\beta, n}$ are uniquely
determined.
\end{theorem}

\begin{remark}
(1) We can include the case $n = 0$ by taking $g_{\beta, 0}$ as an
Eisenstein series $E_\beta$ (\cite[Theorem 1.6]{Br}).

(2) For $k < 0$ one can use \cite[Proposition 1.10]{Br} to construct
$f_{\alpha, m}$ and $g_{\beta, n}$ explicitly.
\end{remark}

If we take $L$ in Theorem \ref{Theorem - Zagier duality for vector
valued modular forms} as a unimodular one, we can recover Guerzhoy's
result \cite[Theorem 1]{Gu}.

\begin{corollary}[Guerzhoy]\label{Corollary - Guerzhoy}
For $m, n \in \mathbb Z_{>0}$ there exist $f_m \in H_k(\SL_2(\mathbb
Z))$, $g_n \in M_{2 - k}^!(\SL_2(\mathbb Z))$ with Fourier
expansions of the form
\begin{eqnarray*}
f_m(\tau) & = & f_m^-(\tau)
+q^{-m} + \sum_{l \in \mathbb Z_{\geq 0}} c_{f_m}^+(l) q^l \\
g_n(\tau) & = & q^{-n} + \sum_{l \in \mathbb Z_{> 0}} c_{g_n}(l) q^l
\end{eqnarray*}
such that
\[ c_{f_m}^+(n) = -c_{g_n}(m). \]
If $k < 0$, then $f_m$, $g_n$ are uniquely determined.
\end{corollary}

Of course we may take $L$ as another suitable lattice. Let $p$ be an
odd prime. We write $M_k^!(\Gamma_0(p), (\frac{\cdot}{p}))$ for the
space of weakly holomorphic modular forms of integral weight $k$ for
$\Gamma_0(p)$ with Nebentypus $(\frac{\cdot}{p})$. For $\epsilon \in
\{ \pm 1 \}$ we define the subspace
\[ M_k^{!\epsilon}(\Gamma_0(p), \big(\frac{\cdot}{p}\big))
:= \{ f = \sum_{n \in \mathbb Z} c_f(n) q^n \in M_k^!(\Gamma_0(p),
\big(\frac{\cdot}{p}\big)) \, | \, c_f(n) = 0 \mbox{ if }
\big(\frac{n}{p}\big) = - \epsilon \}. \]

Suppose that the discriminant group $L'/L$ is isomorphic to $\mathbb
Z / p\mathbb Z$. Then $b^+ - b^-$ is even, and the quadratic form on
$L'/L$ is equivalent to $Q(\gamma) = \frac{\lambda \gamma^2}{p}$ for
$\lambda, \gamma \in \mathbb Z / p \mathbb Z$ with $\lambda \neq 0$.

Now we put $\epsilon = (\frac{\lambda}{p})$, $\delta =
(\frac{-1}{p}) \epsilon$, and assume that $k \equiv (b^+ - b^-)/2
\bmod 2$. Then Bruinier and Bundschuh \cite[Theorem 5]{BB} showed
that the space $M_k^{!\epsilon}(\Gamma_0(p), (\frac{\cdot}{p}))$
(resp. $M_k^{!\delta}(\Gamma_0(p), (\frac{\cdot}{p}))$) of scalar
valued weakly holomorphic modular forms is isomorphic to the space
$M_{k, \rho_L}^!$ (resp. $M_{k, \bar{\rho}_L}^!$) of vector valued
ones.

We observe that Bruinier and Bundschuh's argument can also be
applied to the spaces of scalar and vector valued harmonic weak
Maass forms. Let $H_k (\Gamma_0(p), (\frac{\cdot}{p}))$ denote the
space of harmonic weak Maass forms of weight $k$ for $\Gamma_0(p)$
with Nebentypus $(\frac{\cdot}{p})$ (see Section 2). In particular
$f \in H_k(\Gamma_0(p), (\frac{\cdot}{p}))$ has a unique
decomposition $f = f^+ + f^-$, where
\begin{eqnarray*}
f^+(\tau) & = & \sum_{n \gg -\infty} c_f^+(n) q^n, \\
f^-(\tau) & = & \sum_{n < 0} c_f^-(n) \Gamma(1 - k, 4\pi |n| y) q^n.
\end{eqnarray*}
Here $\Gamma(a, y) = \int_y^\infty e^{-t} t^{a - 1} dt$ denotes the
incomplete Gamma function. For $\epsilon \in \{ \pm 1 \}$ we define
the subspace
\[ H_k^{\epsilon}(\Gamma_0(p), \big(\frac{\cdot}{p}\big))
:= \{ f \in H_k(\Gamma_0(p), \big(\frac{\cdot}{p}\big)) \, | \,
c_f^\pm(n) = 0 \mbox{ if } \big(\frac{n}{p}\big) = - \epsilon \}.
\]

For a given $f \in H_k(\Gamma_0(p), (\frac{\cdot}{p}))$ we define a
$\mathbb C[L'/L]$-valued function $F = \sum_{\gamma \in \mathbb Z /
p\mathbb Z} F_\gamma \mathfrak e_\gamma$ by
\begin{eqnarray*}
F_\gamma(\tau) := s(\gamma) \sum_{n \in \mathbb Z \atop n \equiv p
Q(\gamma) \bmod p} c_f(n, \frac{y}{p}) q^{n/p}.
\end{eqnarray*}
Here $c_f(n, y) := c_f^+(n) + c_f^-(n)\Gamma(1 - k, 4 \pi |n| y)$,
and for $l \in \mathbb Z$
\[ s(l) := \left\{ \begin{array}{ll} 2 & \mbox{if } l \equiv 0 \bmod
p, \\
1 & \mbox{otherwise.} \end{array} \right. \]

\begin{proposition}\label{Proposition - isomorphism for integral weight case}
Let $L$ be a non-degenerate even lattice of signature $(b^+, b^-)$
such that the discriminant group $L'/L$ is isomorphic to $\mathbb Z
/ p\mathbb Z$ with $p$ an odd prime. Put $\epsilon =
(\frac{\lambda}{p})$ and $\delta = (\frac{-1}{p}) \epsilon$. We
assume that $k \equiv \frac{b^+ - b^-}{2} \bmod 2$. Then the map $f
\mapsto F$ defines an isomorphism of $H_k^{\epsilon}(\Gamma_0(p),
(\frac{\cdot}{p}))$ onto $H_{k, \rho_L}$, and of
$H_k^{\delta}(\Gamma_0(p), (\frac{\cdot}{p}))$ onto $H_{k,
\bar{\rho}_L}$.
\end{proposition}

\begin{remark}
A similar result for half integral weight case can be found in
\cite{CC}.
\end{remark}

As another corollary of Theorem \ref{Theorem - Zagier duality for
vector valued modular forms} combining with Proposition
\ref{Proposition - isomorphism for integral weight case} we obtain
Zagier duality between the space $H_k^{\epsilon} (\Gamma_0(p),
(\frac{\cdot}{p}))$ of scalar valued harmonic weak Maass forms and
the space $M_{2 - k}^{!\delta} (\Gamma_0(p), (\frac{\cdot}{p}))$ of
scalar valued weakly holomorphic modular forms.

\begin{corollary}\label{Corollary - Zagier duality for scalar valued modular forms}
With the notation and assumption as in Proposition \ref{Proposition
- isomorphism for integral weight case}, we further assume that $k
\leq 0$. For $m, n \in \mathbb Z_{>0}$ with $(\frac{-m}{p}) \neq
-\epsilon$ and $(\frac{-n}{p}) \neq -\delta$, there exist $f_m \in
H_k^{\epsilon} (\Gamma_0(p), (\frac{\cdot}{p}))$ and $g_n \in M_{2 -
k}^{!\delta} (\Gamma_0(p), (\frac{\cdot}{p}))$ with Fourier
expansions of the form
\begin{eqnarray*}
f_m(\tau) & = & f_m^-(\tau) + q^{-m} + \sum_{l \geq 0} c_{f_m}^+(l)
q^l, \\
g_n(\tau) & = &q^{-n} + \sum_{l > 0} c_{g_n}(l) q^l
\end{eqnarray*}
such that
\[ s(n)c_{f_m}^+(n) = -s(m) c_{g_n}(m). \]
If $k < 0$, then $f_m$, $g_n$ are uniquely determined.
\end{corollary}

\begin{remark}
(1) In \cite{Gu} Guerzhoy called such a pair $(f_m, g_n)$ a
``grid''.

(2) This extends the result given by Rouse \cite[Theorem 1]{RO}.
\end{remark}

\section{\bf{Preliminaries}}

\subsection{Scalar valued modular forms}

Let $\tau = x + iy \in \mathbb H$, the complex upper half plane,
with $x, y \in \mathbb R$. Let $k \in \mathbb Z$, $N$ a positive
integer, and $\chi$ a Dirichlet character modulo $N$.

Recall that \textit{weakly holomorphic modular forms of weight $k$
for $\Gamma_0(N)$ with Nebentypus $\chi$} are holomorphic functions
$f : \mathbb H \rightarrow \mathbb C$ which satisfy:
\begin{enumerate}[(i)]
\item For all $(\begin{smallmatrix} a & b \\ c & d \end{smallmatrix})
\in \Gamma_0(N)$ we have
\[ f\Big(\frac{a\tau + b}{c\tau + d}\Big) = \chi(d)
(c\tau + d)^k f(\tau); \]

\item $f$ has a Fourier expansion of the form
\[ f(\tau) = \sum_{n \in \mathbb Z \atop n \gg -\infty} c_f(n) q^n,
\] and analogous conditions are required at all cusps. Here $q = e^{2 \pi i \tau}$ as
usual.
\end{enumerate}
We write $M_k^!(\Gamma_0(N), \chi)$ for the space of these weakly
holomorphic modular forms.

A smooth function $f : \mathbb H \rightarrow \mathbb C$ is called a
\textit{harmonic weak Maass form of weight $k$ for $\Gamma_0(N)$
with Nebentypus $\chi$} if it satisfies:
\begin{enumerate}[(i)]
\item For all $(\begin{smallmatrix} a & b \\ c & d \end{smallmatrix})
\in \Gamma_0(N)$ we have
\[ f\Big(\frac{a\tau + b}{c\tau + d}\Big) = \chi(d)
(c\tau + d)^k f(\tau); \]

\item $\Delta_k f = 0$, where $\Delta_k$ is the weight
$k$ hyperbolic Laplace operator defined by
\[ \Delta_k := -y^2 \Big( \frac{\partial^2}{\partial x^2} +
\frac{\partial^2}{\partial y^2} \Big) + i k y \Big(
\frac{\partial}{\partial x} + i\frac{\partial}{\partial y} \Big);
\]

\item There is a Fourier polynomial $P_f(\tau) = \sum_{-\infty \ll
n \leq 0} c_f^+(n) q^n \in \mathbb C[q^{-1}]$ such that $f(\tau) =
P_f(\tau) + O(e^{-\varepsilon y})$ as $y \rightarrow \infty$ for
some $\varepsilon > 0$. Analogous conditions are required at all
cusps.
\end{enumerate}
We denote the space of these harmonic weak Maass forms by
$H_k(\Gamma_0(N), \chi)$. This space can be denoted by
$H_k^+(\Gamma_0(N), \chi)$ in the context of \cite{BF}. Here we
follow notation given in \cite{BOR}. Then the space
$H_k(\Gamma_0(N), \chi)$ contains the space $M_k^!(\Gamma_0(N),
\chi)$ of weakly holomorphic modular forms of weight $k$ for
$\Gamma_0(N)$ with Nebentypus $\chi$. The polynomial $P_f \in
\mathbb C[q^{-1}]$ is called the principal part of $f$ at the
corresponding cusps. In particular $f \in H_k(\Gamma_0(N), \chi)$
has a unique decomposition $f = f^+ + f^-$, where
\begin{eqnarray*}
f^+(\tau) & = & \sum_{n \gg -\infty} c_f^+(n) q^n, \\
f^-(\tau) & = & \sum_{n < 0} c_f^-(n) \Gamma(1 - k, 4\pi |n| y) q^n.
\end{eqnarray*}

\subsection{Vector valued modular forms}

We write $\Mp_2(\mathbb R)$ for the metaplectic two-fold cover of
$\SL_2(\mathbb R)$. The elements are pairs $(M, \phi)$, where $M =
(\begin{smallmatrix} a & b \\ c & d
\end{smallmatrix}) \in \SL_2(\mathbb R)$ and $\phi : \mathbb H
\rightarrow \mathbb C$ is a holomorphic function with $\phi(\tau)^2
= c\tau + d$. The multiplication is defined by
\[ (M, \phi(\tau)) (M', \phi'(\tau)) = (MM', \phi(M'\tau)
\phi'(\tau)). \] For $M = (\begin{smallmatrix} a & b \\ c & d
\end{smallmatrix}) \in \SL_2(\mathbb R)$ we use the notation
$\tilde{M} := ((\begin{smallmatrix} a & b \\ c & d
\end{smallmatrix}), \sqrt{c\tau + d}) \in \Mp_2(\mathbb R)$. We
denote by $\Mp_2(\mathbb Z)$ the integral metaplectic group, that is
the inverse image of $\SL_2(\mathbb Z)$ under the covering map
$\Mp_2(\mathbb R) \rightarrow \SL_2(\mathbb R)$. It is well known
that $\Mp_2(\mathbb Z)$ is generated by $T := \left(
(\begin{smallmatrix} 1 & 1 \\ 0 & 1 \end{smallmatrix}), 1 \right)$
and $S := \left( (\begin{smallmatrix} 0 & -1 \\ 1 & 0
\end{smallmatrix}), \sqrt{\tau} \right)$.

Let $(V, Q)$ be a non-degenerate rational quadratic space of
signature $(b^+, b^-)$. Let $L \subset V$ be an even lattice with
dual $L'$. We denote the standard basis elements of the group
algebra $\mathbb C[L' / L]$ by $\mathfrak e_\gamma$ for $\gamma
\in L'/L$, and write $\langle \cdot, \cdot \rangle$ for the
standard scalar product, antilinear in the second entry, such that
$\langle \mathfrak e_\gamma, \mathfrak e_{\gamma'} \rangle =
\delta_{\gamma, \gamma'}$. There is a unitary representation
$\rho_L$ of $\Mp_2(\mathbb Z)$ on $\mathbb C[L'/L]$,  called the
Weil representation, which is defined by
\begin{eqnarray*}
\rho_L(T) (\mathfrak e_\gamma) & := & e(Q(\gamma)) \mathfrak
e_\gamma, \\
\rho_L(S) (\mathfrak e_\gamma) & := & \frac{e((b^- -
b^+)/8)}{\sqrt{|L'/L|}} \sum_{\delta \in L'/L} e(-(\gamma, \delta))
\mathfrak e_\delta,
\end{eqnarray*}
where $e(z) := e^{2 \pi i z}$ and $(X, Y) := Q(X + Y) - Q(X) - Q(Y)$
is the associated bilinear form. We denote by $\bar{\rho}_L$ the
dual representation of $\rho_L$.

Let $k \in \frac{1}{2}\mathbb Z$. A holomorphic function $f :
\mathbb H \rightarrow \mathbb C[L'/L]$ is called a \textit{weakly
holomorphic modular form of weight $k$ and type $\rho_L$ for the
group $\Mp_2(\mathbb Z)$} if it satisfies:
\begin{enumerate}[(i)]
\item $f(M\tau) = \phi(\tau)^{2k} \rho_L(M, \phi) f(\tau)$ for all
$(M, \phi) \in \Mp_2(\mathbb Z)$;

\item $f$ is meromorphic at the cusp $\infty$.
\end{enumerate}
Here condition (ii) means that $f$ has a Fourier expansion of the
form
\[ f(\tau) = \sum_{\gamma \in L'/L} \sum_{n \in \mathbb Z + Q(\gamma)
\atop n \gg -\infty} c_f(\gamma, n) e(n\tau) \mathfrak e_\gamma. \]
The space of these $\mathbb C[L'/L]$-valued weakly holomorphic
modular forms is denoted by $M_{k, \rho_L}^!$. Similarly we can
define the space $M_{k, \bar{\rho}_L}^!$ of $\mathbb C[L'/L]$-valued
weakly holomorphic modular forms of type $\bar{\rho}_L$. The
subspace of $\mathbb C[L'/L]$-valued holomorphic modular forms
(resp. cusp forms) of weight $k$ and type $\rho_L$ is denoted by
$M_{k, \rho_L}$ (resp. $S_{k, \rho_L}$).

A smooth function $f : \mathbb H \rightarrow \mathbb C[L'/L]$ is
called a \textit{harmonic weak Maass form of weight $k$ and type
$\rho_L$ for the group $\Mp_2(\mathbb Z)$} if it satisfies:
\begin{enumerate}[(i)]
\item $f(M\tau) = \phi(\tau)^{2k} \rho_L(M, \phi) f(\tau)$ for all
$(M, \phi) \in \Mp_2(\mathbb Z)$;

\item $\Delta_k f = 0$;

\item There is a Fourier polynomial $P_f(\tau) = \sum_{\gamma \in L'/L}
\sum_{n \in \mathbb Z + Q(\gamma) \atop -\infty \ll n \leq 0}
c_f^+(\gamma, n) e(n\tau) \mathfrak e_\gamma$ such that $f(\tau) =
P_f(\tau) + O(e^{-\varepsilon y})$ as $y \rightarrow \infty$ for
some $\varepsilon > 0$.
\end{enumerate}
We denote by $H_{k, \rho_L}$ the space of these $\mathbb
C[L'/L]$-valued harmonic weak Maass forms. This space is denoted by
$H_{k, L}^+$ in \cite{BF}. We have $M_{k, \rho_L}^! \subset H_{k,
\rho_L}$. Similarly we define the space $H_{k, \bar{\rho}_L}$. In
particular $f \in H_{k, \rho_L}$ has a unique decomposition $f = f^+
+ f^-$, where
\begin{eqnarray*}
f^+(\tau) & = & \sum_{\gamma \in L'/L} \sum_{n \in \mathbb Z + Q(\gamma)
\atop n \gg -\infty} c_f^+(\gamma, n) e(n\tau) \mathfrak e_\gamma, \\
f^-(\tau) & = & \sum_{\gamma \in L'/L} \sum_{n \in \mathbb Z +
Q(\gamma) \atop n < 0} c_f^-(\gamma, n) \Gamma(1 - k, 4\pi |n| y)
e(n\tau) \mathfrak e_\gamma.
\end{eqnarray*}

\subsection{Zagier duality for weakly holomorphic modular
forms}\label{Zagierduality}

We begin with the following rather simple observation.

\begin{lemma}\label{Lemma - observation}
Let $k, k' \in \frac{1}{2}\mathbb Z$ with $k + k' \in \mathbb Z$,
and
\begin{eqnarray*}
f & = & \sum_{\gamma \in L'/L} f_\gamma \mathfrak e_\gamma \in
M_{k, \rho_L}^!, \\
g & = & \sum_{\gamma \in L'/L} g_\gamma \mathfrak e_\gamma \in
M_{k', \bar{\rho}_L}^!.
\end{eqnarray*}
Then $\sum_{\gamma\in L'/L} f_\gamma g_\gamma \in M_{k +
k'}^!(\SL_2(\mathbb Z))$, i.e. a weakly holomorphic elliptic modular
form of weight $k + k'$.
\end{lemma}

\begin{proof}
Since $\sum_{\gamma\in L'/L} f_\gamma(\tau) g_\gamma(\tau)$ is
holomorphic on $\mathbb H$ and meromorphic at the cusp $\infty$, it
suffices to verify the modular transformation property. First note
that
\[ \sum_{\gamma\in L'/L} f_\gamma g_\gamma = \langle \sum_{\gamma\in L'/L}
f_\gamma \mathfrak e_\gamma, \sum_{\gamma\in L'/L}
\overline{g_\gamma} \mathfrak e_\gamma \rangle = \langle f,
\overline{g} \rangle. \] Here $\langle \cdot , \cdot \rangle$
denotes the usual dot product.

\noindent   Now we find for $M = (\begin{smallmatrix} a & b \\
c & d
\end{smallmatrix}) \in \SL_2(\mathbb Z)$ that
\begin{eqnarray*}
\big( \sum_{\gamma\in L'/L} f_\gamma g_\gamma \big) \big|_{k + k'} M
& = & (c\tau + d)^{-(k + k')} \sum_{\gamma\in L'/L} f_\gamma(M\tau)
g_\gamma(M\tau) \\
& = & \langle (c\tau + d)^{-k} f(M\tau), \overline{(c\tau +
d)^{-k'} g(M\tau)} \rangle \\
& = & \langle \rho_L(\tilde{M}) f(\tau), \rho_L(\tilde{M})
\overline{g(\tau)} \rangle.
\end{eqnarray*}
Since $\rho_L$ is unitary we have
\[ \big( \sum_{\gamma\in L'/L} f_\gamma g_\gamma \big) \big|_{k +
k'} M = \langle f(\tau), \overline{g(\tau)} \rangle =
\sum_{\gamma\in L'/L} f_\gamma g_\gamma. \]
\end{proof}

\begin{lemma}\label{Lemma - residue theorem}
Let $k, k' \in \frac{1}{2}\mathbb Z$ with $k + k' = 2$. We denote
the Fourier expansions of $f \in M_{k, \rho_L}^!$ and $g \in M_{k',
\bar{\rho}_L}^!$ by
\begin{eqnarray*}
f(\tau) & = & \sum_{\gamma\in L'/L} \sum_m c_f(\gamma, m) q^m
\mathfrak e_\gamma, \\
g(\tau) & = & \sum_{\gamma\in L'/L} \sum_n c_g(\gamma, n) q^n
\mathfrak e_\gamma. \\
\end{eqnarray*}
Then
\[ \sum_{\gamma\in L'/L} \sum_{m + n = 0} c_f(\gamma, m) c_g(\gamma, n) = 0. \]
\end{lemma}

\begin{proof}
By Lemma \ref{Lemma - observation}, $\sum_{\gamma\in L'/L} f_\gamma
g_\gamma$ is a weakly holomorphic elliptic modular form of weight
$2$. Thus by the residue theorem we immediately obtain the
assertion.
\end{proof}

\begin{remark}\label{za}
\begin{enumerate}
\item
 In Lemmas \ref{Lemma - observation} and \ref{Lemma - residue
theorem} the Weil representation $\rho_L$ can be replaced by
arbitrary unitary representation because we only used the unitary
property of the Weil representation.

\item Lemma \ref{Lemma - residue theorem} is indeed a special case
of Bruinier and Funke's result (see Proposition \ref{Proposition -
bilinear pairing} below). This can be used to show Zagier duality
for weakly holomorphic modular forms of integral or half integral
weight. For example one can easily infer \cite[Theorem 4]{Za},
\cite[Theorem 1]{RO} from Lemma \ref{Lemma - residue theorem} and
the results \cite[Theorems 5.1, 5.4]{EZ}, \cite[Theorem 5]{BB}
concerning the isomorphism between the spaces of scalar and vector
valued modular forms.

\end{enumerate}
\end{remark}

\section{\bf{The results of Bruinier, Funke, Ono, and Rhoades \cite{BF, BOR}}}

Assume that $k \leq 0$ is an integer. Recall that there is an
antilinear differential operator
\[ \xi_k : H_{k, \rho_L} \longrightarrow S_{2 - k,
\bar{\rho}_L} \] defined by
\[ \xi_k(f)(\tau) := y^{k - 2} \overline{L_k f(\tau)}, \]
where $L_k := -2iy^2 \frac{\partial}{\partial \bar{\tau}}$ is the
Maass lowering operator (see \cite{BF, BOR}). The Maass raising
operator is defined by $R_k := 2i \frac{\partial}{\partial \tau} + k
y^{-1}$. By \cite[Corollary 3.8]{BF} the following sequence is exact:
\[ \begin{CD}
0   @>>>   M_{k, \rho_L}^!   @>>>   H_{k, \rho_L} @>\xi_k>> S_{2 -
k, \bar{\rho}_L}   @>>>   0.
\end{CD} \]

\begin{proposition}\cite[Proposition 3.11]{BF}\label{Proposition - prescribed principal part}
For every Fourier polynomial of the form
\[ P(\tau) = \sum_{\gamma \in L'/L} \sum_{n \in \mathbb Z + Q(\gamma)
\atop n < 0} c^+(\gamma, n) e(n\tau) \mathfrak e_\gamma \] with
$c^+(\gamma, n) = (-1)^{k +  \frac{b^- - b^+}{2}} c^+(-\gamma,
n)$, there exists an $f \in H_{k, \rho_L}$ with principal part
$P_f(\tau) = P(\tau) + \mathfrak c$ for some $T$-invariant
constant $\mathfrak c \in \mathbb C[L'/L]$. The function $f$ is
uniquely determined if $k < 0$.
\end{proposition}

Using the Petersson scalar product, we can obtain a bilinear pairing
between $M_{2 - k, \bar{\rho}_L}$ and $H_{k, \rho_L}$ defined by
\[ \{ g, f \} = (g, \xi_k(f)) := \int_{\SL_2(\mathbb Z) \backslash
\mathbb H} \langle g, \xi_k(f) \rangle y^{2 - k} \frac{dx dy}{y^2},
\] where $g \in M_{2 - k, \bar{\rho}_L}$ and $f \in H_{k, \rho_L}$.

\begin{proposition}\cite[Proposition 3.5]{BF}\label{Proposition - bilinear pairing}
If $g \in M_{2 - k, \bar{\rho}_L}$ and $f \in H_{k, \rho_L}$, then
\[ \{ g, f \} = \sum_{\gamma \in L'/L} \sum_{n \leq 0} c_f^+(\gamma,
n) c_g(\gamma, -n). \]
\end{proposition}

As pointed out in Remark \ref{za}, one can easily infer Zagier
duality for weakly holomorphic modular forms from Proposition
\ref{Proposition - bilinear pairing}. But this is insufficient for
our purpose because we are going to deal with harmonic weak Maass
forms. To this end we need to define the regularized bilinear
pairing $\{ \cdot , \cdot \}^{reg}$.

Following Borcherds \cite[Section 6]{Bo}, we define the regularized
bilinear pairing $\{ g, f \}^{reg}$ between $M_{2 - k,
\bar{\rho}_L}^!$ and $H_{k, \rho_L}$ as the constant term in the
Laurent expansion at $s = 0$ of the meromorphic continuation in $s$
of the function
\[ \lim_{t \rightarrow \infty} \int_{\mathcal F_t} \langle g, \xi_k(f)
\rangle y^{2 - k - s} \frac{dx dy}{y^2}, \] where
\[ \mathcal F_t = \{ \tau \in \mathbb H \, | \, |\tau| \geq 1, \,
|x| \leq \frac{1}{2}, \mbox{ and } y \leq t \} \] denotes the
truncated fundamental domain for $\SL_2(\mathbb Z)$ on $\mathbb H$.

\begin{proposition}\label{Proposition - regularized bilinear pairing}
Let $f \in H_{k, \rho_L}$ and $g \in M_{2 - k, \bar{\rho}_L}^!$.
Suppose that $g$ has vanishing constant terms, i.e. $c_g(\gamma, 0)
= 0$ for all $\gamma \in L'/L$. Then
\[ \{ g, f \}^{reg} = \sum_{\gamma \in L'/L} \sum_{m + n = 0}
c_f^+(\gamma, m) c_g(\gamma, n). \]
\end{proposition}

\begin{proof}
According to \cite[Remark 8]{BOR} we have
\[ \{ g, f \}^{reg} = \lim_{t \rightarrow \infty} \int_{\mathcal F_t}
\langle g, \xi_k(f) \rangle y^{2 - k} \frac{dx dy}{y^2}. \] Now one
can apply the same argument as in the proof of \cite[Proposition
3.5]{BF}.
\end{proof}

Recall the differential operator
\[ D := \frac{1}{2 \pi i} \frac{d}{d\tau}. \]
From Bol's identity \cite[Lemma 2.1]{BOR} one finds
\[ D^{1 - k} = \frac{1}{(-4\pi)^{1 - k}} R_k^{1 - k}. \]

\begin{proposition}\label{Proposition - from Bol's identity}
If $h \in H_{k, \rho_L}$, then
\[ D^{1 - k} h \in M_{2 - k, \rho_L}^!. \]
Moreover, we have
\[ D^{1 - k} h = D^{1 - k} h^+ = \sum_{\gamma, n} c_h^+(\gamma, n)
n^{1 - k} e(n\tau) \mathfrak e_\gamma. \]
\end{proposition}

\begin{proof}
Note that $R_k^{1 - k} h$ satisfies the transformation behavior for
vector valued modular forms of weight $2 - k$ and type $\rho_L$. The
remaining assertions can be derived by the same argument as in the
proof of \cite[Theorem 1.1]{BOR}.
\end{proof}

\begin{proposition}\label{Proposition - orthogonality}
If $f \in H_{k, \rho_L}$ and $h \in H_{k, \bar{\rho}_L}$, then
\[ \{ D^{1 - k} h, f \}^{reg} = 0. \]
\end{proposition}

\begin{proof}
Note that $\xi_k (f) \in S_{2 - k, \bar{\rho}_L}$ is a cusp form.
Now the assertion follows from \cite[Corollary 4.2]{BOR}.
\end{proof}

Combining Propositions \ref{Proposition - regularized bilinear
pairing}-\ref{Proposition - orthogonality} we have the following.

\begin{theorem}\label{Theorem - a general form of Zagier duality}
Let $k \leq 0$ be an integer, $f \in H_{k, \rho_L}$, and $g \in D^{1
- k} (H_{k, \bar{\rho}_L}) \subset M_{2 - k, \bar{\rho}_L}^!$. Then
\[ \sum_{\gamma \in L'/L} \sum_{m + n = 0} c_f^+(\gamma, m)
c_g(\gamma, n) = 0. \]
\end{theorem}

\section{\bf{Proof of Theorem \ref{Theorem - Zagier duality for vector
valued modular forms}}}

Let $\alpha, \beta \in L'/L$ and $m \in \mathbb Z - Q(\alpha)$, $n
\in \mathbb Z + Q(\beta)$ with $m, n > 0$. Then by Proposition
\ref{Proposition - prescribed principal part} there exist
$f_{\alpha, m} \in H_{k, \rho_L}$ and $h_{\beta, n} \in H_{k,
\bar{\rho}_L}$ such that
\begin{eqnarray*}
f_{\alpha, m}(\tau) & = & f_{\alpha, m}^-(\tau) + q^{-m} \mathfrak
e_\alpha + (-1)^{k + \frac{b^- - b^+}{2}} q^{-m} \mathfrak
e_{-\alpha} + \sum_{\gamma \in L'/L} \sum_{l \in \mathbb Z +
Q(\gamma) \atop l
\geq 0} c_{f_{\alpha, m}}^+(\gamma, l) q^l \mathfrak e_\gamma, \\
h_{\beta, n}(\tau) & = & h_{\beta, n}^-(\tau) + q^{-n} \mathfrak
e_\beta + (-1)^{k + \frac{b^+ - b^-}{2}} q^{-n} \mathfrak
e_{-\beta} + \sum_{\gamma \in L'/L} \sum_{l \in \mathbb Z -
Q(\gamma) \atop l \geq 0} c_{h_{\beta, n}}^+(\gamma, l) q^l
\mathfrak e_\gamma.
\end{eqnarray*}
By Proposition \ref{Proposition - from Bol's identity} the
$g_{\beta, n} := (-n)^{k - 1} D^{1 - k} h_{\beta, n} \in M_{2 - k,
\bar{\rho}_L}^!$ has a Fourier expansion of the form
\[ g_{\beta, n}(\tau)  = q^{-n} \mathfrak e_\beta + (-1)^{k + \frac{b^+ - b^-}{2}}
q^{-n} \mathfrak e_{-\beta} + \sum_{\gamma \in L'/L} O(q^\nu)
\mathfrak e_\gamma \] for some $\nu > 0$. Now Theorem \ref{Theorem -
Zagier duality for vector valued modular forms} follows from Theorem
\ref{Theorem - a general form of Zagier duality}.

\section{\bf{Proof of Proposition \ref{Proposition - isomorphism for integral
weight case}}}

By the same argument as in \cite{BB} one can see that $F$ satisfies
the desired modular transformation property. To check $\Delta_k F =
0$ we recall from \cite[Proposition 2]{BB} that
\[ F = i^{\frac{b^+ - b^-}{2}} p^{\frac{k - 1}{2}} \sum_{M \in \Gamma_0(p)
\backslash \SL_2(\mathbb Z)} (\rho_L(M)^{-1} \mathfrak e_0) f|_k W_p
|_k M, \] where $W_p := (\begin{smallmatrix} 0 & -1 \\ p & 0
\end{smallmatrix})$ is the Fricke involution. Since $\Delta_k$
commutes with the Petersson slash operator (see \cite{Ma}), we have
$\Delta_k F = 0$. For the converse we recall from \cite[Lemma 1]{BB}
that the inverse isomorphism $F \mapsto f$ is given by
\[ f = \frac{i^{\frac{b^+ - b^-}{2}}}{2} p^{\frac{1 - k}{2}} F_0 |_k W_p. \]
Hence $\Delta_k f$ vanishes.

\bibliographystyle{amsplain}

\end{document}